\documentclass{amsart}
\usepackage{k2}
\usepackage{graphicx}
\usepackage[pdftex]{hyperref}
\hypersetup{colorlinks,citecolor=blue,linktocpage,hyperindex=true,backref=true}
\usepackage[group-separator={,}]{siunitx}
\usepackage{caption}
\numberwithin{equation}{section}

\usepackage{todonotes}

\newcommand\can{_{\rm can}}
\newcommand\bry{\llc B\lrc}

\title{Open surfaces of small volume}
\author{Valery Alexeev and Wenfei Liu}
\date{January 31, 2017}
\begin{document}
\maketitle
\begin{abstract}
  We construct a surface with log terminal singularities and ample
  canonical class that has $K_X^2=1/\num{48 983}$ and a log canonical
  pair $(X,B)$ with a nonempty reduced divisor $B$ and ample $K_X+B$
  that has $(K_X+B)^2 = 1/462$. Both examples significantly improve
  known records. 
\end{abstract}
\tableofcontents

\section{Introduction}

Let $U$ be a smooth quasiprojective surface, and let $S$ be a smooth
compactification such that $D=S\setminus U$ is a normal crossing divisor.
The open surface $U$ is said to be of general type if
$K_S+D$ is big. This condition and the spaces of pluri log canonical
sections $H^0(n(K_S+D))$ for all $n\ge0$ depend only on $U$ and not on
the choice of a particular normal crossing compactification $(S,D)$. 

Since $K_S+D$ is big, the number of its sections grows quadratically:
$h^0(n(K_S+D)) \sim c n^2/2$.  After passing to the log canonical
model $(S\can, D\can)$ where $K_{S\can}+D\can$ is ample, one sees that
$c= (K_{S\can}+D\can)^2$ and it is called the \emph{volume} of the
pair $(S,D)$, equivalently the volume of $U$, and is denoted by
$\vol(K_S+D)= \vol(U)$.

Vice versa, if $(X,B)$ is a log canonical pair with reduced boundary $B$ and ample $K_X+B$, and if
$f\colon S\to X$ is its log resolution with exceptional divisors
$\{E_i\}$ then 
\[
\vol\big(S \setminus (f\inv(B)\cup E_i)\big) = 
 \vol(K_S+f_*\inv B + \sum E_i)=\vol(K_X + B)=(K_X+B)^2.
 \] 

\begin{question}
  How small could a volume of an open surface $U$ of general type be?
  Equivalently, 
  how small could $(K_X+B)^2$ be for a log canonical pair with reduced
  boundary $B$ and ample $K_X+B$?
\end{question}

The basic result in this direction is the following quite general

\begin{theorem}[Alexeev, \cite{alexeev1994boundedness-and-ksp-2}]
  Let $(X, B=\sum b_iB_i)$ be a log canonical pair with coefficients
  $b_i$ belonging to a DCC set $\cS\subset[0,1]$ (that is, a set satisfying
  descending chain condition). Then the set of squares $(K_X+B)^2$ is
  also a DCC set. In particular, it has a minimum, a positive number
  real number, rational if $\cS\subset\bQ$.
\end{theorem}

We will denote this minimum by $K^2(\cS)$. Some interesting DCC sets are
$\cS_0=\emptyset$, $\cS_1=\{1\}$, $\cS_2'=\{1-\frac1{n},\ n\in\bN\}$,
$\cS_2=\cS'_2 \cup \{1\}$. The paper
\cite{alexeev2004bounding-singular} gives an effective, computable
lower bound for $K^2(\cS)$ in terms of the set $\cS$, which is however
too small to be realistic, cf. Section~\ref{sec:k2-lower-bound} where
we spell it out for the sets $\cS_0$ and $\cS_1$. 

A version of the above definition is to look at the pairs $(X,B)$ with
nonempty reduced part of the boundary $\bry\ne0$. We will denote
the minimum in this case by $K^2_1(\cS)$. Clearly, one has
\begin{displaymath}
  K^2(\cS_2) \le K^2(\cS_1) \le K^2(\cS_0)
  \quad\text{and}\quad
  K^2(\cS_1) \le K^2_1(\cS_1).
\end{displaymath}

Some published bounds for $K^2(\cS)$ and $K_1^2(\cS)$ for the
above sets include:
\begin{enumerate}
\item $K^2(\cS_0)\le \frac3{55}$.
\item $K^2_1(\cS_1) \le \frac1{60}$.
\item $K^2_1(\cS_2) = \frac1{42^2}=\frac1{1764}$, and thus 
  $K^2_1(\cS_1) \ge\frac1{\num{1764}}$.
\item $K^2(\cS_2)\le \frac1{(42\cdot 43)^2}$ and the bound is
  conjectured to be sharp.
\end{enumerate}
The first of these bounds follows from an example of Blache
\cite{blache1995example-concerning}. The others are due to Koll\'ar
\cite{kollar1994log-surfaces, kollar2013moduli-varieties}. 

There are also many other examples of log terminal surfaces with ample
$K_X$ appearing in the literature. For example, hypersurfaces
$S(a_1,a_2,a_3,a_4)$ of the form
$x_1^{a_1}x_2 + x_2^{a_2}x_3 + x_3^{a_3}x_4 + x_4^{a_4}x_1=0$ in
weighted projective spaces $\bP(w_1,w_2,w_3,w_4)$ provide such
examples under some mild conditions on the $a_i$'s. These surfaces
were studied in \cite{orlik1977structure-weighted,
  Kouchnirenko1976polyedres-Newton, kollar2008is-there,
  hwang2012construction-singular,
  urzua2016characterization-Kollar}. The last three papers also study
surfaces $S^*(a_1,a_2,a_3,a_4)$ obtained by contracting two curves on
such surfaces, as in the last section of \cite{kollar2008is-there}.

These papers are not specifically concerned with the minimal possible
value of $K_X^2$, but certainly better bounds than $\frac3{55}$ can be
achieved. Jos\'e Ignacio Y\'a\~nez informed us that the following
example appears to achieve the minimum among the surfaces
$S(a_1,a_2,a_3,a_4)$ with $\gcd(w_1,w_3)=\gcd(w_2,w_4)=1$. 
\begin{example}[Urz\'ua-Y\'a\~nez] \label{ex:urzua} The hypersurface
  $S(2,2,4,10)$ of degree $159$ in the weighted projective space
  $\bP(49,61,37,11)$ has an ample canonical class $K_X$, cyclic
  quotient singularities, and 
  \begin{displaymath}
    K_X^2 = \frac{159 \cdot (159 - 49 - 61 - 37 - 11)^2}{49\cdot 61\cdot
      37\cdot 11} = \frac{159}{\num{1216523}} \approx
    \frac1{\num{7651}}.
  \end{displaymath}
\end{example}

\medskip

Knowing the exact bounds is important for many
applications. As explained in \cite{kollar1994log-surfaces}, a stable
limit of surfaces of general type of volume $d$ has at most $d/K^2_1(\cS_1)$
irreducible components. Thus, as a corollary of Koll\'ar's bound
$K^2_1(\cS_1) \ge\frac1{\num{1764}}$ the number of irreducible components is 
at most $\num{1764}\, d$. 
Other applications include bounds for the automorphism groups of surfaces
and surface pairs of general type, see
e.g. \cite{kollar1994log-surfaces, alexeev1994boundedness-and-ksp-2} for
more discussion.

The constant $K^2(\cS_1)$ is certainly a very fundamental global
invariant in its own right: the smallest volume of a smooth open
surface.

\medskip
The main result of this paper is the following:

\begin{theorem}
  One has $K^2_1(\cS_1)\le\frac1{462}$, and $K^2(\cS_1) = K^2(\cS_0) \le
  \frac1{\num{48 983}}$.
\end{theorem}

Section~\ref{sec:method} explains the method which we used to find the
new examples. We restate it in Section~\ref{sec:game} as a purely
combinatorial game with weights $(w_0,w_1,w_2,w_3)$.
Section~\ref{sec:weights1} contains some easy instances of this game
for the simplest weights $(0,1,1,1)$, $(1,1,1,1)$ giving in particular
Koll\'ar's example with $(K_X+B)^2=\frac1{60}$.
Sections~\ref{sec:with-boundary} and \ref{sec:empty-boundary} contain
our champion examples for the winning weights $(1,2,3,5)$: 
surfaces with a nonempty boundary and
$(K_X+B)^2=\frac1{462}=\frac1{11\cdot 42}$, and 
surfaces without boundary and $K_X^2=\frac1{\num{48
    983}}=\frac1{11\cdot 61\cdot 73}$.

In characteristic 0 the champion surfaces have Picard number
$\rho(X)=2$ and they have 4 (resp. 3) singularities. But in
characteristic 2 the rank of the Picard group drops by 1 and there is
an additional $A_1$-singularity. The surfaces with ample $K_X$ and such
configurations of singularities would provide counterexamples to the
algebraic Montgomery-Yang problem, were they to exist in characteristic
0. We discuss this connection in Section~\ref{sec:montgomery-yang}.

Related to this, in Section~\ref{sec:pic1} we list some surfaces of
small volume that have Picard rank $\rho(X)=1$. 
In particular, we prove that for the surfaces $S^*(a_1,a_2,a_3,a_4)$
with $\gcd(w_1,w_3)=\gcd(w_2,w_4)$ the minimum is
$K_X^2=\frac1{6351}$. 

Section~\ref{sec:four-lines} explains why we restricted to the case of
only four lines in our search. 
Finally, Section~\ref{sec:k2-lower-bound} spells out the effective
(but very small) bound for $K_X^2$ which follows from
\cite{alexeev2004bounding-singular,  kollar1994log-surfaces}. 

\medskip

Further, we note that Koll\'ar's lower bound for $K^2_1(\cS_2)$ is a
combination of two inequalities. One defines two invariants of a DCC
set $\cS$:

\begin{definition}\label{def:epsilon-delta} ${}$
\begin{enumerate}
\item Let $(X,B)$ be a log canonical surface with ample $K_X+B$ and
  $\bry\ne0$. Then  
\[
\epsilon_1(X,B):=\min_{B_0\subset\bry}\{(K_X+B)B_0\}
\]
and $\epsilon_1(\cS)$ is the minimum of these numbers as $(X,B)$ go over
  all pairs with coefficients in $\cS$.
\item $\delta_1(\cS)$ is the minimum of $t>0$ such that there exists a
  log canonical pair $(X, (1-t)B_0 + \Delta)$ with
  $K_X+(1-t)B_0+\Delta\equiv 0$ such that the coefficients of $\Delta$
  are in $\cS$.
\item We also define a closely related invariant of an individual big log
  canonical  divisor:
  $\delta_1(X,B)$ is the minimum of $t$ such that $K_X+B-tB_0$ is
  not big for some $0\ne B_0\subset\bry$, or 1 if this minimum is $>1$.
\end{enumerate}  
\end{definition}

Then according to \cite{kollar1994log-surfaces} one has
$(K_X+B)^2\ge \epsilon_1(\cS)\delta_1(\cS)$ and
$\epsilon_1(\cS_2)= \delta_1(\cS_2)=\frac1{42}$. In
Section~\ref{sec:with-boundary} we give an example that shows that the
equality $\epsilon_1(\cS_1)=\frac1{42}$ also holds. As for
$\delta_1(\cS_1)$, we were not able to find better than $\frac1{13}$
with the present method, which is the same as in Koll\'ar's example
with $(K_X+B_0)^2=\frac1{60}$.

\medskip

All constructions and examples in this paper work over an
algebraically closed field of arbitrary characteristic.

\begin{acknowledgements}
  The work of the first author was partially supported by NSF under
  DMS-1603604. He would like to thank Giancarlo Urz\'ua and Jos\'e
  Ignacio Y\'a\~nez for helpful discussions.  The second author was
  supported by NSFC (No.~11501012) and by the Recruitment Program for
  Young Professionals. He would like to thank S\"onke Rollenske and
  Stephen Coughlan for helpful discussions about log canonical
  surfaces with small volumes. Especially S\"onke Rollenske helped to
  transform Koll\'ar's example to $\bP^2$ with four lines, as
  illustrated in Section~\ref{sec:weights1}.
\end{acknowledgements}

\section{The method of construction}
\label{sec:method}

We begin with four lines $L_0, L_1, L_2, L_3$ in $\bP^2$ in general
position. Let $f\colon \wX\to \bP^2$ be a sequence of blowups, each at a
point of intersection of two divisors that appeared so far:
exceptional divisors, lines, and their strict preimages. We will call
thus obtained divisors on $\wX$ the \emph{visible curves}.
We will assume that enough blowups were performed so that the
strict preimages of lines satisfy $L_k^2\le -1$. Thus, all visible
curves will have negative self-intersection. 

Let $E_i$ be the visible curves with $E_i^2\le -2$ and $C_j$ be the
visible curves with $C_j^2=-1$. Now assume that the curves $\{E_i\}$
form a \emph{log terminal configuration}, i.e., a configuration of 
exceptional curves on the  minimal resolution of a surface with log 
terminal singularities. Each connected component of the dual graph
is of type $A_n$ (i.e. a chain)
with no further restrictions on self-intersections $E_i^2$, and one of the
graphs of types $D_n$ and $E_n$, with restrictions on
self-intersections, see e.g.~\cite{alexeev1992log-canonical-surface}.

Log terminal configurations are rational and,  by Artin \cite{artin1962some-numerical}, the curves $E_i$ on $\tilde X$ can be contracted to obtain a projective surface $X$.  Let
$\pi\colon \wX\to X$ be the contraction morphism.  
(More generally, one may assume that
$\{E_i\}$ form a log canonical configuration. 
In our examples, only log terminal singularities occur.)

The surface $\wX$ is then the minimal resolution of singularities of~$X$ and one has
\begin{displaymath}
  K_\wX = \pi^* K_X + \sum a_i E_i, \quad 
  \pi^* K_X = K_\wX+\Delta := K_\wX + \sum b_i E_i.
\end{displaymath}
Here, $a_i$ are the \emph{discrepancies} and $b_i=-a_i$.  The numbers
$b_i$ satisfy the following linear system of $n$ equations in $n$
variables:
\begin{equation}\label{eq:discrepancies}
  (K_\wX + \sum b_iE_i)E_j = 0 \iff 
  \sum_i b_i E_iE_j= E_j^2+2 \quad \text{for any } j
\end{equation}
By Mumford, the matrix $(E_iE_j)$ is negative-definite, so this system
has a unique solution.  By Artin \cite{artin1962some-numerical}, all
entries of the matrix $(E_iE_j)\inv$ are $\le 0$. Since the right-hand
sides are $E_j^2+2 \le0$, it follows that $b_i\ge0$.
\cite{alexeev1992log-canonical-surface} contains some convenient
closed-form formulas for $b_i$'s, see also \cite{miyanishi2001open-algebraic}.

\begin{theorem}\label{thm:no-bdry}
  Let $C_j$ be the visible $(-1)$-curves on $\wX$. Assume that:
  \begin{enumerate}
  \item For all $C_j$ one has $K_X\pi_*(C_j) \ge 0$ (resp. $K_X\pi_*(C_j) > 0$).
  \item $K_X^2 > 0$.
  \item There exist four rational numbers $d_0,d_1,d_2,d_3$ with
  $\sum d_k=3$ such that the coefficients of $C_j$ in the formula below
  are all $d_j\le 0$ (resp. all $d_j<0$):
  \begin{displaymath}
    K_\wX + \wD := f^*(K_{\bP^2} + \sum d_k L_k) = K_\wX + \sum d_i
    E_i + \sum d_j C_j.
  \end{displaymath}
  \end{enumerate}
  Then the divisor $K_X$ is big and nef (resp. ample).  
\end{theorem}
\begin{proof}
  Of course, (1) and (2) are necessary for $K_X$ to be big and
  nef (resp. ample). Condition (3) implies that $K_X$ is an effective
  linear combination of the curves $\pi(C_j)$. Indeed, 
  $K_\wX+\wD = f^*(0)=0$, so $K_X = \pi_*(-\wD) = \sum (-d_j)
  \pi_*(C_j)$. So, $K_X$ intersects any irreducible curve on $X$
  non-negatively. Thus, $K_X$ is nef.

  For ampleness, note that union of visible curves supports an
  effective ample divisor. Thus, any curve on $X$ intersects its
  image, $\cup \pi(C_j)$. Therefore, any irreducible curve on $X$
  intersects $\sum (-d_j)\pi_*(C_j)$ positively, and so $K_X$ is ample by
  Nakai-Moishezon criterion.
\end{proof}

\begin{remark}
  Even if $K_X$ is only big and nef, it is semiample by Abundance
  Theorem in dimension 2, so its canonical model has ample $K_{X\can}$
  and the same square $K_{X\can}^2=K_X^2$. 
\end{remark}

\begin{remark}
  We usually assume that $0\le d_k \le 1$ (so that $\sum d_kL_k$ is a
  ``boundary'' in the standard MMP terminology), or at least that
  $d_k\le 1$ (so that it is a ``sub boundary''). But this is not
  necessary for the above proof.
\end{remark}

The coefficients $d_i$, $d_j$ in the divisor $\wD$ for the visible
curves $E_i,C_j$ are readily computable. The formula is especially
simple in terms of the quantities $(1-d_i)$, which are just the
\emph{log discrepancies} of $(\bP^2, \sum d_kL_k)$: after blowing up
the point of intersection of two curves with log discrepancies
$1-d_1$, $1-d_2$, the new log discrepancy is $1-d_3 =
(1-d_1)+(1-d_2)$. In other words, the log discrepancies add up.

The above will be our essential method for finding new examples in the
case of the empty boundary. For the examples with a nonempty boundary,
we do not contract the strict preimage of the line $L_0$, which we
denote by $\wB_0$.  We no longer include $\wB_0$ in either collections
$\{E_i\}$, $\{C_j\}$. 
We modify the assumption made at the beginning of this Section to
allow $\wB_0^2$ to be non-negative, since it is a ``special'' curve. Let
$B_0$ be the image of $\wB_0$ on $X$. Then we define the discrepancies
for the pair $(X,B_0)$ via

\begin{displaymath}
  \pi^* (K_X+B_0) = K_\wX+\Delta := K_\wX + \wB_0 + \sum b_i E_i.
\end{displaymath}

With this modification, Theorem~\ref{thm:no-bdry} readily extends:

\begin{theorem}\label{thm:bdry}
  Let $C_j$ be the visible $(-1)$-curves on $\wX$. Assume that:
  \begin{enumerate}
  \item For all $C$ in $\{\pi(C_j),B_0\}$ one has $(K_X+B_0)C \ge 0$
    (resp. $(K_X+B_0)C > 0$).
  \item $(K_X+B_0)^2 > 0$.
  \item There exist four rational numbers $d_0,d_1,d_2,d_3$ with
  $\sum d_k=3$ such that the coefficients of $C_j$ in the formula below
  are all $d_j\le 0$ (resp. all $d_j<0$):
  \begin{displaymath}
    K_\wX + \wD := f^*(K_{\bP^2} + \sum d_k L_k) = K_\wX + d_0\wB_0 + \sum d_i
    E_i + \sum d_j C_j
  \end{displaymath}
  In addition, assume that $d_0\le 1$ (resp. $d_0<1$). 
  \end{enumerate}
  Then the divisor $K_X+B_0$ is big and nef (resp. ample).  
\end{theorem}

\begin{proof}
  The same proof as in Theorem~\ref{thm:no-bdry} gives that
  $K_X+ d_0B_0$ is an effective (resp. strictly positive) combination
  of the curves $\pi(C_j)$. But then so is
  $K_X+B_0 = K_X+d_0B_0 + (1-d_0)B_0$. If $d_0<1$ then $B_0$ appears
  in this sum with a positive coefficient. The rest of the proof is
  the same.
\end{proof}

We now state without proof some easy formulas.

\begin{lemma}\label{lem:EK-K2}
  The following hold.
  (For the surface without a boundary, omit $B_0$.)
  \begin{enumerate}
  \item $(K_X+B_0) \pi_*(C_j)=(K_\wX+\Delta)C_j = -1 + \wB_0 C_j + \sum_i
    b_i E_iC_j$. 
  \item
    $(K_X+B_0)^2= (K_\wX+\Delta)^2 = K_\wX^2 + K_\wX\Delta +
    (\Delta-\wB_0)\wB_0 - 2.$
    For the surface without a boundary,
    $K_X^2= (K_\wX+\Delta)^2 = K_\wX^2 + K_\wX\Delta.$
  \item $K_{\wX}^2 = 9 - (\text{the number of blowups in } \wX\to\bP^2)$.
  \item $K_\wX E_i= -E_i^2-2$ and $K_\wX \wB_0 = -\wB_0^2-2$.
  \end{enumerate}
\end{lemma}

\section{Combinatorial game}
\label{sec:game}

As usual, we associate with a configuration of curves on a surface its
dual graph. The vertices are labeled with marks $-E_i^2$ (we call them
\emph{marks} because we use \emph{weights} for a different
purpose). Thus, the initial configuration of lines on $\bP^2$
corresponds to a complete graph on four vertices with marks
$-1, -1, -1, -1$, and we have a graph describing the visible curves on
$\wX$.  To simplify the typography, the $(-1)$-curves $C_j$ are shown
in white with no marks. The exceptional curves $E_i$ are shown in black,
and the marks $2$ are omitted.  The reduced boundary $\wB_0$, if
present, is shown as a crossed vertex.

We call the dual graph of the visible curves on $\wX$ the
\emph{visible graph}.  It can be obtained by performing a series of
insertions in the initial graph on four vertices with given marks
$-1,-1,-1,-1$.  Each instance is an insertion of a new vertex $v_3$
with mark 1 between two vertices $v_1, v_2$, at the same time
increasing the marks of $v_1$ and $v_2$ by 1.

Now we attach a weight to each vertex of the visible graph. First we
choose rational numbers $w_0, w_1, w_2, w_3$ for the four vertices of the
initial graph, called the \emph{initial weights}. We define the
weights for the other vertices inductively in the process of inserting
vertices as follows.  As a new vertex $v_3$ is inserted between two
vertices $v_1, v_2$ with already assigned weights $w(v_1)$ and
$w(v_2)$, we define the weight of $v_3$ to be $w(v_3)=w(v_1)+w(v_2)$.

Indeed, our weights are just the suitably normalized log discrepancies
for the pair $(\bP^2, \sum_{k=0}^3 d_k L_k)$: $w_s = n(1- d_s)$ for
some positive rational number $n$, for all the visible curves. 
We can always rescale $n$ to make $w_s$ integers, if we like.

\begin{lemma}\label{lem:weight-conds}
  In terms of the weights, the conditions on the coefficients
  $d_k,d_0,d_j$  in 
  Theorems~\ref{thm:no-bdry} and \ref{thm:bdry} translate to the
  following:
\begin{enumerate}
\item $d_0+d_1+d_2+d_3=3$ $\iff$ $n=w_0+w_1+w_2+w_3$.
\item $d_j\le 0$ (resp. $d_j<0$) for a visible $(-1)$-curve $C_j$
  $\iff$ the weight $w_j\ge n$ (resp. $w_j>n$) for the corresponding
  white vertex. 
\item In the case with a nonempty boundary, $d_0\le 1$ (resp. $d_0<1$)
  for the curve $\wB_0$ $\iff$ $w_0\ge0$ (resp. $w_0>0$) for the
  corresponding crossed vertex. 
\end{enumerate}  
\end{lemma}

For as long as the weights satisfy these conditions, all
we have to do is this:

\begin{enumerate}
\item Make sure that the configuration of the black vertices $\{E_i\}$
  is log terminal. 
\item Compute the negatives of the discrepancies $b_i$ from the 
  linear system~\ref{eq:discrepancies}, or using the formulas from
  \cite{alexeev1992log-canonical-surface}, or by any other method.
\item Make sure that $K_X\,\pi_*(C_j)\ge0$, $K_X\wB_0>0$ (if $\wB_0$
  is present), and $K_X^2>0$ (resp. $(K_X+B_0)^2>0$) using the
  formulas in Lemma~\ref{lem:EK-K2}.
\end{enumerate}
If all of these are satisfied then we get ourselves an example of a
log canonical surface (which is either $X$ or $X\can$) with ample
(log) canonical divisor.

\medskip
 
The visible graph is homeomorphic to a complete graph on four
vertices, but the edges between the corners may contain many
intermediate vertices.

\begin{definition}
  We call an edge a \emph{Calabi-Yau (or CY) edge} if all the white vertices on
  this edge have weights $n$, that is, they have discrepancies $d_i=0$. 
\end{definition}

The picture below shows two examples of edges that we use. 
  \begin{center}
  \begin{tikzpicture}
    \begin{scope}[every node/.style={draw,color=black,inner
        sep=2pt,fill,circle}] 
      \node[fill=white] (a0) at (0,0) {};
      \node (a1) at (1,0) {};
      \node (a2) at (2,0) {};
      \node (a3) at (3,0) {};
      \node (a4) at (4,0) {};
      \end{scope}
      \draw[very thin] (a0)--(a1);
      \draw[ultra thick] (a1)--(a2) (a3)--(a4);
      \draw[dashed] (a2)--(a3);
  \end{tikzpicture}
  \end{center}
  \begin{center}
  \begin{tikzpicture}
    \begin{scope}[every node/.style={draw,color=black,inner
        sep=2pt,fill,circle}] 
      \node[fill=white] (a0) at (0,0) {};
      \node (a1) at (1,0) {};
      \node (a2) at (2,0) {};
      \node (a3) at (3,0) {};
      \node (a4) at (4,0) {};
      \node (a5) at (5,0) {};
      \node (a6) at (6,0) {};
      \node (a7) at (7,0) {};
      \node[fill=white] (a8) at (8,0) {};
      \end{scope}
      \node[above of=a4, node distance=1em] {3};
      \draw[very thin] (a0)--(a1) (a7)--(a8);
      \draw[ultra thick] (a1)--(a2) (a3)--(a4)--(a5) (a6)--(a7);
      \draw[dashed] (a2)--(a3) (a5)--(a6);
  \end{tikzpicture}
  \end{center}
  In general, if the end vertices have weights $w_1$, $w_2$ then
  vertices in the interior of this chain have weights of the form
  $m_1w_1+m_2w_2$ for some coprime positive integers $m_1,m_2$, and
  the way in which these integers are produced in the sequence of
  blowups is equivalent to the well known in number theory
  Stern-Brocot tree. Thus, every edge is encoded by a sequence of
  positive rational numbers $\{m_{1,i}/m_{2,i}\}$ for the white
  vertices. For the two examples above, the sequences are $\{m/1\}$
  and $\{m/1, 1/m'\}$. For a CY edge, these numbers must satisfy
  the condition $m_{1,i}w_1+m_{2,i}w_2 = n$.



\medskip
The following two lemmas are very easy and are given without proof. 

\begin{lemma}\label{lem:near-CY}   ${}$
  \begin{enumerate}
  \item Suppose that all edges in the visible graph are CY edges.  If
    $B_0\ne0$ then in addition suppose that the weight $w_0=0$.  Then
    one has $K_X\equiv 0$ (resp. $K_X+B_0\equiv 0$). 
  \item Suppose that all edges are CY, except for one where there is a
    unique white vertex of weight $n+1$. If $B_0\ne0$ then in addition
    suppose that the weight $w_0=0$.  Then $K_X \equiv \frac1{n}C$,
    where $C$ is an image of a $(-1)$-curve corresponding to that
    special white vertex (resp. $K_X+B_0\equiv\frac1{n}C$).
  \item In the case with the nonempty boundary, suppose that all edges are CY
    and that the weight of $L_0$ is 1. Then $K_X+B_0\equiv \frac1{n}B_0$
    and $K_X \equiv -\frac{n-1}{n} B_0$.
  \end{enumerate}
\end{lemma}

\begin{lemma}\label{lem:delta-eps-computed}
  In the last case (3) of Lemma~\ref{lem:near-CY}, the invariant
  $\delta_1$ defined in \eqref{def:epsilon-delta} equals
  $\delta_1(X,B_0)=\frac1{n}$.  In all cases with nonempty boundary,
  one has $\epsilon_1(X,B_0) = -2 + \sum b_i E_i\wB_0$.
\end{lemma}

Our main strategy for finding interesting examples is this: start with
a CY situation, i.e., case (1) of Lemma~\ref{lem:near-CY}, and then go
just one step up, to get in the situation of cases (2) or (3). The
(log) canonical divisor is then guaranteed to be quite small.

\begin{remark}
  The surfaces $X$ appearing in case (1) of Lemma~\ref{lem:near-CY}
  with empty boundary satisfy $K_X\equiv 0$ and $H^1(\cO_X)=0$, and
  are sometimes called log Enriques surfaces. Our construction with
  weights provides a huge supply of such surfaces.
\end{remark}

\begin{remark}
  Another convenient way to compute $(K_X+B)^2$, resp. $K_X^2$, is the
  following. If $K_X+B = \frac1{n}C$ then certainly
  $(K_X+B)^2 = \frac1{n} (K_X+B) C$. So one has to maximize the sum of
  the weights $n$ and to minimize $(K_X+B)C$.
\end{remark}

\section{Simplest weights $(0,1,1,1)$ and $(1,1,1,1)$}
\label{sec:weights1}

Figure~\ref{fig:60-462} gives the smallest volumes that can be achieved
playing the combinatorial game with $B\ne0$ and weights $(0,1,1,1)$,
and with $B=0$ and weights $(1,1,1,1)$.


\begin{figure}[!h]
  \includegraphics{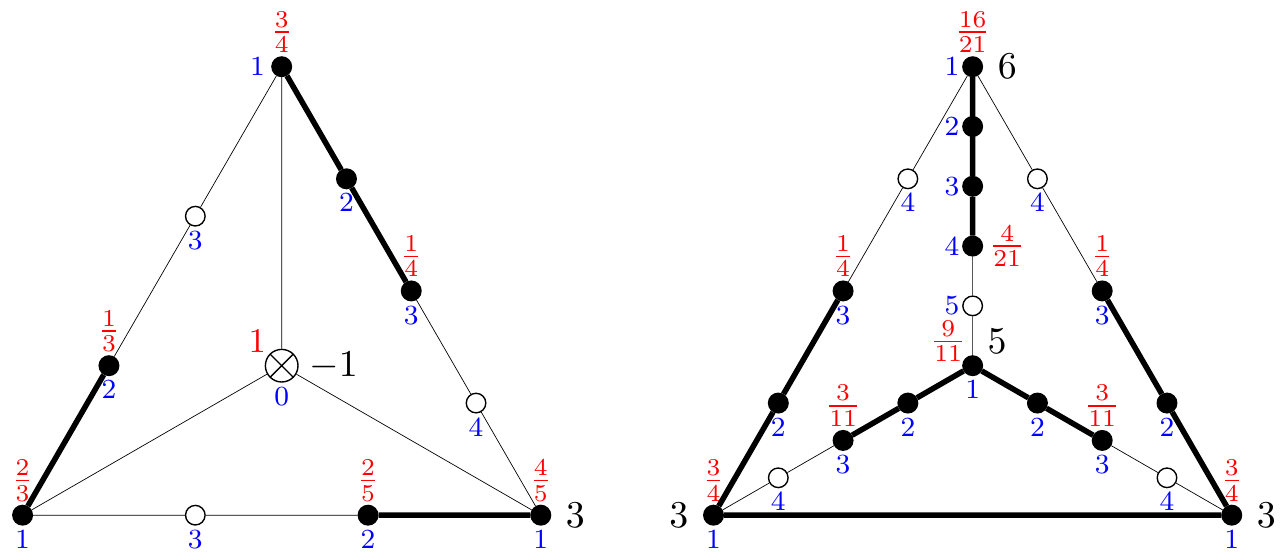}
  \caption{Surfaces with $(K_X+B)^2=\frac1{60}$ and
    $K_X^2=\frac1{462}$}
  \label{fig:60-462}
\end{figure}

\begin{notation}
  The large numbers are marks, when not equal to 1 or 2. The blue
  small numbers underneath are the weights, and the red fractional
  numbers on top are $b_i$, the negatives of discrepancies.
\end{notation}

For the first pair $(X,B)$ one has
$\epsilon_1=1-\frac13-\frac14-\frac15 = \frac{13}{60}$.  There is an
alternative choice of weights $(1,3,4,5)$ in this case, for which this
pair fits into the case (3) of Lemma~\ref{lem:near-CY}, i.e. all edges
are CY. Then $\delta_1=\frac1{13}$ by
Lemma~\ref{lem:delta-eps-computed}, and
$(K+B)^2 = \frac{13}{60}\cdot\frac1{13} = \frac1{60}$. In fact,
another description for this pair is $(X,B) = (\bP(3,4,5), D_{13})$
where $D_{13}$ is a degree 13 weighted hypersurface, and this is
exactly Koll\'ar's example from \cite{kollar2013moduli-varieties}. One
has $\rho(X)=\rank\Pic X=1$, and $K_X+B$, $B$, and $-K_X$ are ample.

For the second pair we can also choose the weights $(1-2\epsilon, 1,
1+\epsilon, 1+\epsilon)$, $0<\epsilon<\frac12$, for which
\eqref{thm:no-bdry}, \eqref{lem:weight-conds}
give that $K_X$ is ample and not just big
and nef.

\section{Example with nonempty boundary: $1/{462}$}
\label{sec:with-boundary}

Figure~\ref{fig:462s} shows two non-isomorphic visible graphs producing
pairs $(X,B)$ of the smallest volume with nonempty boundary that we
were able to find by our method.

\begin{figure}[!h]
  \includegraphics{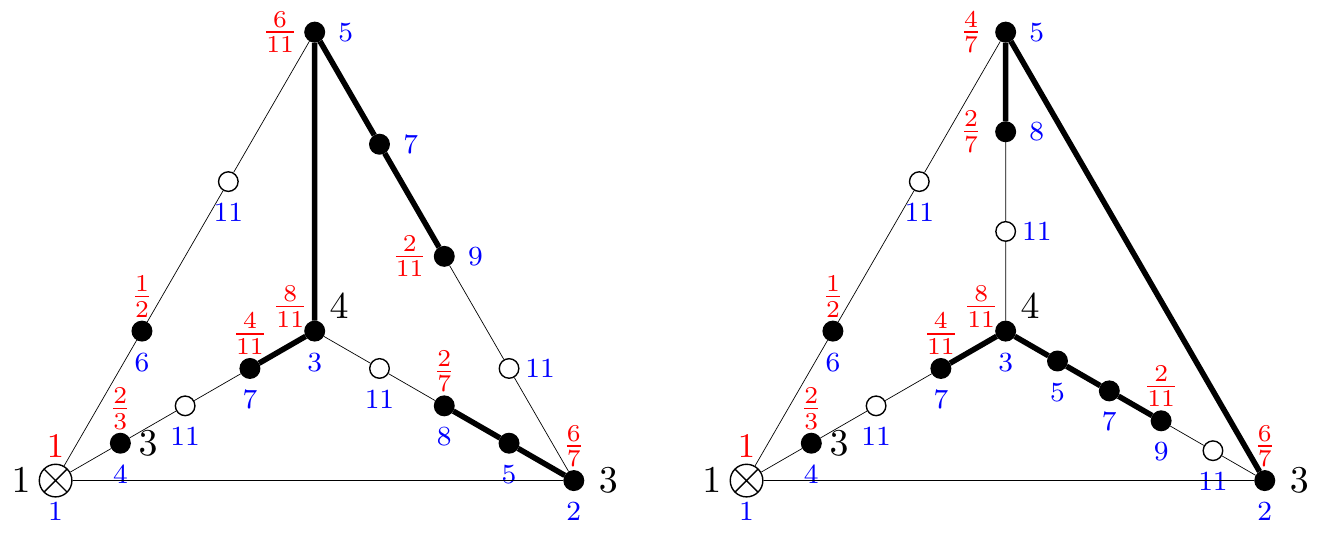} 
  \caption{Surfaces with $(K_X+B)^2=\frac1{462}=\frac1{11\cdot 42}$ and $B\ne0$}
  \label{fig:462s}
\end{figure}

The winning weights are $(1,2,3,5)$. Both pairs achieve the minimal
possible value of
$\epsilon_1 = 1 - \frac12 - \frac13 - \frac17 = \frac1{42}$. On the
other hand, one has $K_X+B \equiv \frac{1}{11}B$ and
$\delta_1=\frac1{11}$ by Lemma~\ref{lem:delta-eps-computed}. Thus, one
has $(K_X+B)^2 = \epsilon_1 \delta_1$.  The divisor $K_X+B$ is big and
nef but there are no weights for which \eqref{thm:no-bdry},
\eqref{lem:weight-conds} or a variation of them show that $K_X+B$ is
ample. The canonical model has ample $K_{X\can}+B\can$ and
$(K_{X\can}+B\can)^2 = (K_X+B)^2 = \frac1{462}$.

\begin{notation}
  We use Figure~\ref{fig:462s} for an alternative labeling of the
  curves, as follows.  We denote the strict preimages of the four
  lines by $\wL^w$, where the superscript $w=1,2,3,5$ is the (blue)
  weight of the corresponding vertex.  Similarly, we denote one of the
  remaining curves
  by $\wF_w^{i,j}$ if its vertex lies on the edge
  between $\wL^i,\wL^j$ and has weight $w$.  In the same vein, we call
  the initial lines in the plane $L^w$ and the intersection points
  $P^{i,j} = L^i\cap L^j$.

  Finally, we use this labeling for the standard orthogonal basis of
  $\Pic\wX$ consisting of the pullback $H$ of a line from $\bP^2$ and
  the full preimages of the $(-1)$-curves from the intermediate
  blowups $\wX\to\bP^2$.  Thus, in $\Pic\wX$ we have
  $\wL^1 = H - F^{1,3}_4 - F^{1,5}_6$,
  $\wF^{1,3}_4 = F^{1,3}_4 - F^{1,3}_7 - F^{1,3}_{11}$,
  $\wF^{1,3}_7 = F^{1,3}_7 - F^{1,3}_{11}$,
  $\wF^{1,3}_{11} = F^{1,3}_{11}$, etc.
\end{notation}

\begin{theorem}\label{thm:is_K+B_ample}
  The two visible graphs of Figure~\ref{fig:462s} describe the same
  surface $X$.  If $\chr k\ne2$ then $K_X+B$ is ample, $\rho(X)=2$,
  and $X=X\can$ has 4 singularities. If $\chr k=2$ then $K_X+B$ is
  big, nef, but not ample, and it contracts a $(-2)$-curve. One has
  $\rho(X\can)=~1$, and $X\can$ has 5 singularities, the last one a
  simple $A_1$.
\end{theorem}

\begin{proof}

  The second graph has an extra $(-1)$-curve $F^{1,5}_{11}$ not
  present in the first graph. It is easy to see that with respect
  to the first graph it is simply
  the strict preimage of the line in $\bP^2$ connecting the points
  $P^{1,5}$ and $P^{2,3}$. Thus, the surfaces are the same but
  different curves are illuminated as being visible.

  Let us work with the first representation.  Suppose that there
  exists another, not visible curve $\wD$ such that
  $\wD\cdot\pi^*K_X=0$, which is then contracted by a linear system
  $|N\pi^* K_X|$ for $N\gg0$. Since $\rho(X)=2$, there could only be
  one such irreducible curve. We write
  \begin{math}
    \wD = d H - \sum m^{ij}_w F^{ij}_w.
  \end{math}
  The divisor $\wD$ intersects by zero the curves $\wL^1$,
  $\wF^{1,3}_4$, $\wF^{1,5}_6$, $\wL^2$, $\wF^{2,3}_5$, $\wF^{2,3}_8$
  since the full pullback $\pi^* K_X=\frac1{11}\pi^* B_0$ is a
  strictly positive combination of these curves. Also, $\wD$
  intersects non-negatively all the other visible curves.  This gives
  an explicit set of identities and inequalities. One checks that the
  only solution~is
  \begin{displaymath}
    \wD = c \big[ 3H - (2F^{1,3}_4 + F^{1,3}_7 + F^{1,3}_{11} )
    - (F^{1,5}_6 + F^{1,5}_{11}) 
    - (F^{2,5}_7 + F^{2,5}_9 + F^{2,5}_{11}) \big].
  \end{displaymath}
  Then $\wD K_\wX=0$ and $\wD^2=-2c^2$. It follows from the genus formula
  that $c=1$, and $\wD$ is a smooth rational $(-2)$-curve. It must be
  a strict preimage of a cubic curve $D$ in $\bP^2$ which has:
  \begin{enumerate}
  \item a cusp at $P^{1,3}$ with the tangent direction $L^3$,
  \item a flex at $P^{2,5}$ with the tangent direction $L^2$,
  \item a tangent at $P^{1,5}$ to the line $L^5$. 
  \end{enumerate}
  Thus, $D$ is a cuspidal cubic, and the set of points of
  $D\setminus P^{1,3}$ has the structure of the additive group
  $\bG_a$. We see that the cubic $D$ with the above properties exists
  if and only if the system of equations
  $3P^{2,5}=P^{2,5} + 2P^{1,5}=0$ has a solution in the base field $k$
  with $P^{2,5}\ne P^{1,5}$.  This is possible iff $\chr k=2$; then
  $P^{2,5}=0$ and $P^{1,5}$ is any other smooth point.  This completes
  the proof.

  A second proof using the alternative presentation of surface $X$
  works in characteristics 2 and 0, and by extension in all but
  finitely many other positive characteristics.  The second visible
  graph of Figure~\ref{fig:462s} leads to a smooth rational
  $(-2)$-curve
  \begin{displaymath}
    \wD = 4H - (2F^{1,3}_4 + F^{1,3}_7 + F^{1,3}_{11} )
    - (2F^{1,5}_6 + 2F^{1,5}_{11}) 
    - (F^{2,3}_5 + F^{2,3}_7 + F^{2,3}_9 + F^{2,3}_{11}).
  \end{displaymath}
  It must be then a strict preimage of a quartic curve $D$ on $\bP^2$
  that has:
  \begin{enumerate}
  \item an $A_2$-singularity (a cusp) at $P^{1,3}$ with the
    tangent line $L^3$,
  \item an $A_3$ (a tachnode) or $A_4$-singularity at
    $P^{1,5}$ with the tangent line $L^5$,
  \item a hyperflex at $P^{2,3}$ with the tangent direction $L^2$,
    i.e. $L^2\cap D$ is a smooth point of $D$ and the intersection is
    of multiplicity $4$. 
  \end{enumerate}
  If $\chr k=0$ then such quartic curves do not exist. Indeed,
  \cite[Table 2]{wall1995geometry-quartic} shows that irreducible
  quartic curves in characteristic 0 with $A_2A_3$ or $A_2A_4$
  singularities do not have any hyperflexes.
  Since the property of $K_X+B$ being ample is open, the same is true
  in all but finitely many prime characteristics.
  On the other hand, if $\chr k=2$ then there exists a unique such
  curve (with $A_2A_3$), which can be concluded from the normal forms of
  quartics given in \cite{wall1995quartic-curves}. Explicitly, the
  equation of $D$ can be taken to be $f=x^4 + x^2yz + x^2y^2 + y^2z^2$
  and the lines are $L^1=(x)$, $L^2=(y+z)$, $L^3=(x+z)$, $L^5=(y)$.
\end{proof}

\begin{remark}
  Combining the two presentations of surface $X$ in the proof, we see
  that a quartic with the above configuration of singularities and
  tangent lines does not exist in prime characteristics $p\ne2$.
\end{remark}

\begin{remark}\label{lem:ints_w_visible}
  The intersections of $\wD$ with the visible curves are zero except
  for the following:
  \begin{enumerate}
  \item For the first graph,
    $\wD\wF^{1,3}_{11}=\wD\wF^{1,5}_{11}=\wD\wF^{2,3}_{11}=1$. 
  \item For the second graph, $\wD\wF^{1,5}_{11}=2$ and
    $\wD\wF^{1,3}_{11}=\wD\wF^{2,3}_{11}=1$.
  \end{enumerate}
\end{remark}

\section{Example with empty boundary: $1/{\num{48 983}}$}
\label{sec:empty-boundary}

There exist at least four visible graphs that produce surfaces $X$
without a boundary and
$K_X^2 = \frac1{\num{48983}} = \frac1{11\cdot 61\cdot 73}$.  Three of
them share the same list of singularities; the list is different for
the fourth graph.  The first two graphs can be obtained directly by
inserting 10 vertices into the edge from $\wL^1$ to $\wL^2$, i.e. by
blowing up the surfaces $\wX$ of Figure~\ref{fig:462s} above the point
$P^{1,2}$ 10 times.  We show one of these surfaces in
Figure~\ref{fig:48983}. It leads to a surface with $\rho(X)=2$ and
three singularities.

The fourth graph describes a surface $X'$ with
$\rho(X')=3$ which has four singularities: the singularity with the
minimal resolution $[2,4,2,2,2]$ of Figure~\ref{fig:48983} is replaced
by two singularities $[2,6]$ and $[2,3,2,2]$.

\begin{figure}[!h]
  \includegraphics{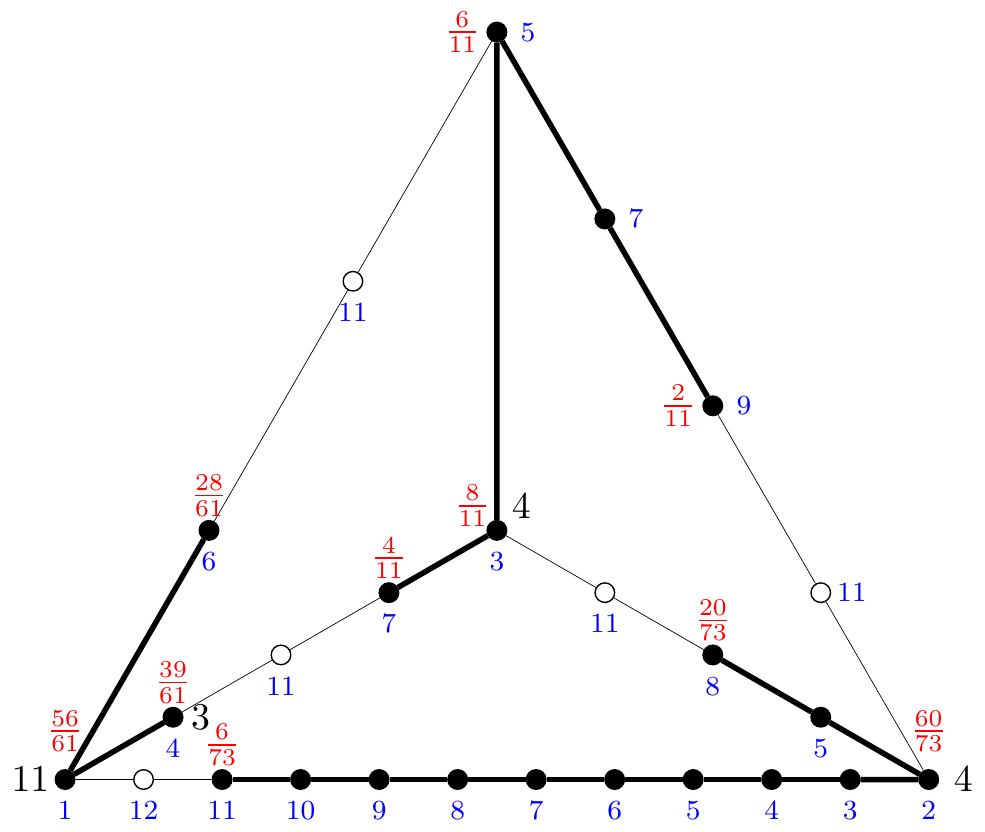}
  \caption{Surface with $K_X^2=\frac1{\num{48983}}=\frac1{11\cdot 
      61\cdot 73}$}
  \label{fig:48983}
\end{figure}

The set of winning weights in these cases is again $(1,2,3,5)$.  Since
$\num{48983} > 42^2$, these examples show that $K^2(\cS_1)$ is achieved
when the boundary is empty, that is,
$K^2_1(\cS_1) > K^2(\cS_1) = K^2(\cS_0)$.


Similarly to the previous case, the divisor $K_X$ is big and nef but
there are no weights for which \eqref{thm:no-bdry},
\eqref{lem:weight-conds} or a variation of them show that $K_X$ is
ample. But the canonical model $X\can$ has ample canonical class and
$K_{X\can}^2=K_X^2 = \frac1{\num{48983}}$.

\begin{theorem}\label{thm:is_K_ample}
  The three distinct visible graphs describe the same surface $X$. The
  fourth graph describes a different surface $X'$ which however has the
  same canonical model since there is a crepant blow down $X'\to X$
  contracting the image of a $(-1)$-curve from $\wX'$. 
  If $\chr k\ne2$ then $K_X$ is ample, $\rho(X)=2$,
  and $X=X\can$ has 3 singularities. If $\chr k=2$ then $K_X$ is
  big, nef, but not ample, and it contracts a $(-2)$-curve;
  $\rho(X\can)=~1$, and $X\can$ has 4 singularities, the last one a
  simple $A_1$.
\end{theorem}

\begin{proof}
  The proof of the equivalence for the first three graphs is the same
  as in Theorem~\ref{thm:is_K+B_ample}. Since the surface in the end
  is unique we do not draw the other two graphs but indicate the two
  invisible curves that have to be added to Figure~\ref{fig:48983} to
  obtain them. These are the strict preimages of a line in $\bP^2$
  joining $P^{1,5}$ and $P^{2,3}$, and of a conic passing through
  $P^{1,2}$ generically, through $P^{1,5}$ with the tangent $L^5$, and
  through $P^{2,3}$ with the tangent $L^3$.

  Similarly, the surface $X'$ described by a fourth graph, which we do
  not draw, has an invisible curve $C$, a strict preimage of a line
  through $P^{1,2}$ and $P^{3,5}$ such that $C\cdot
  \pi^*K_{X'}=0$. Contracting this curve gives the same surface as
  in Figure~\ref{fig:48983}.

  From now on, we work with the surface $X$ described by
  Figure~\ref{fig:48983}.  Again, if there exists an invisible curve
  $\wD$ with $\wD\cdot \pi^*K_X=0$ then it must have zero intersection
  with the curves effectively supporting $\pi^*K_X$, which include the
  10 newly inserted curves $\wF^{12}_w$. Thus, the inequalities in
  this case are reduced to those in Theorem~\ref{thm:is_K+B_ample},
  and the rest of the proof is the same.
\end{proof}

As in Remark~\ref{lem:ints_w_visible}, the intersections of $D$ with
the visible curves are zero except for those listed there. 



\section{Connection with the algebraic Montgomery-Yang problem} 
\label{sec:montgomery-yang}

The algebraic Montgomery-Yang problem
\cite[Conj. 30]{kollar2008is-there} asks whether there exists a surface
with $\rho(S)=1$ and $\pi_1(S\setminus\Sing S)=1$ that has
\emph{four} quotient singularities. Conjecturally, the answer is no.
All the possibilities for such surfaces were ruled out except when
$K_S$ is ample, see \cite{hwang2012construction-singular}.

It is amusing to note that if the characteristic 2 surface $S=X\can$
with 4 singularities which we constructed in
Theorem~\ref{thm:is_K_ample} existed in characteristic 0 then it
would provide a counterexample to the above conjecture.

Let $U_s\ni s$ be a small neighborhood of a singular point $s$ and
$L_s=U_s\setminus s$. Then the three singularities whose determinants
$m=11,61,73$ are coprime to $2$ are quotient singularities and one has
$\pi_1^{\rm alg}(L_s) = \bZ_m$. For the fourth singularity obtained by
contracting a $(-2)$-curve one has $\pi_1^{\rm alg}(L_s) = 1$ in
characteristic 2. One can prove that
$\pi_1^{\rm alg}(S\setminus\Sing S)=1$ by the usual methods, by
considering the images of the $(-1)$-curves $C_j$ connecting the
singularities and using van Kampen theorem, which still
holds for the \'etale fundamental group in positive characteristic by
\cite[IX, Th.5.1]{grothendieck1971sga1-2}, cf. 
\cite{mathoverflow110511}.

In any case, this surface also violates the orbifold
Bogomolov-Miyaoka-Yau inequality $c_1^2(S) \le 3e_{\rm orb}(S)$, for
which one may see the discussion in \cite[\S1]{kollar2008is-there}.
Namely, it violates its corollary, the inequality
\begin{displaymath}
  \sum_{s\in\Sing S} \left(1 - \frac1{|\pi_1(U_s\setminus s)|}\right)
  \le 3, 
\end{displaymath}
if one literally replaces $\pi_1$ with $\pi_1^{\rm alg}$, or if one
replaces $|\pi_1(U_s\setminus s)|$ with $m_i$. Thus,
this configuration of singularities can not appear in characteristic 0
if $\rho(X)=1$. 

\section{The case of Picard rank 1}
\label{sec:pic1}

In part because of the connection with the algebraic Montgomery-Yang
problem, it is of interest to know the minimal volume for surfaces
with the additional condition $\rho(X)=1$, in characteristic 0. For
surfaces without the boundary, the best we were able to find is
$K_X^2 = \frac1{6351} = \frac1{3\cdot 29\cdot 73}$.

There are three possible graphs for the visible curves in this case,
and one of them describes a surface that can be obtained by
contracting two curves on the hypersurface in a weighted projective
space from Example~\ref{ex:urzua}. Thus, it is one of the surfaces
$S^*(2,2,4,10)$ studied in \cite[Sec.43]{kollar2008is-there}.

The other two graphs are not of this type, and we give one of them in
Figure~\ref{fig:6351}. However, all three graphs share the same list
of singularities. Indeed, the argument we gave in the proof of
Theorems~\ref{thm:is_K+B_ample}, \ref{thm:is_K_ample} shows that the
three visible graphs describe the same surface.

\begin{figure}[!h]
  \includegraphics{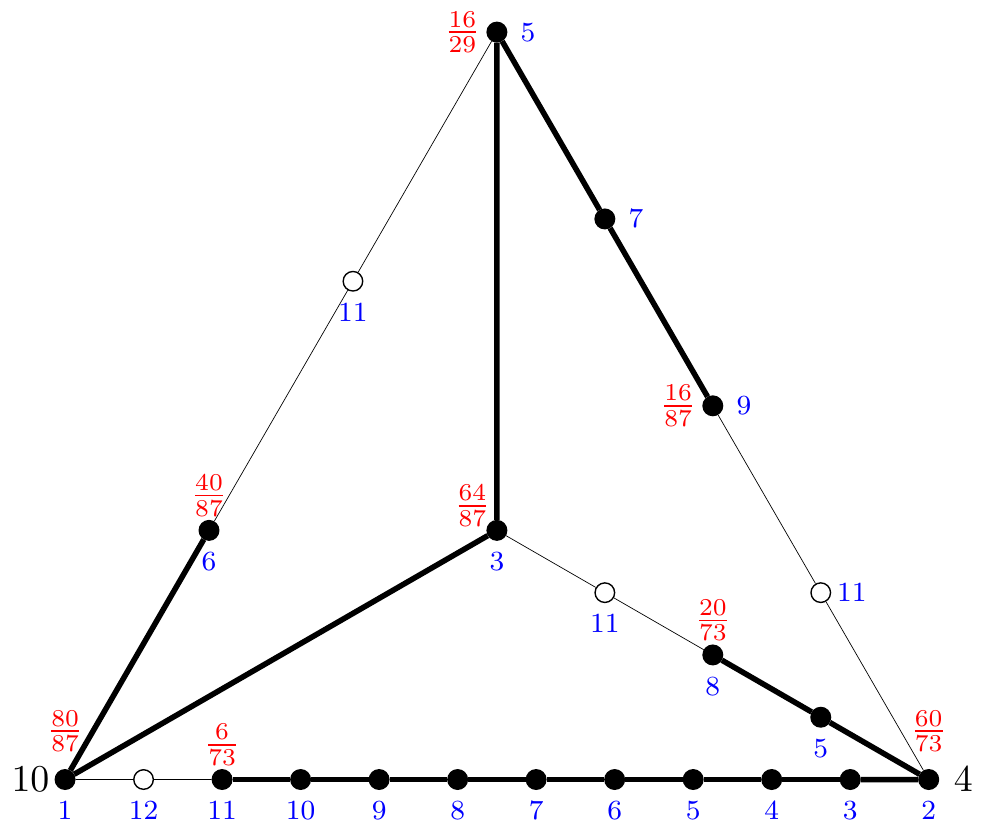}
  \caption{Surface with $K_X^2=\frac1{6351}=\frac1{3\cdot 29\cdot 73}$
  and $\rho(X)=1$}
  \label{fig:6351}
\end{figure}

\begin{remark}
  Comparing Figures~\ref{fig:48983} and \ref{fig:6351}, one can see
  that our surface with the minimal volume can be obtained from the
  surface $S^*(2,2,4,10)$ by a weighted blowup at one point. 
\end{remark}

In \cite{hwang2012construction-singular} Hwang and Keum construct,
for any $a_1,a_2,a_3,a_4\ge2$, a surface $T=T(a_1,a_2,a_3,a_4)$ with
$\rho(T)=1$ obtained by blowing up the 4-line configuration; it has
two cyclic singularities corresponding to the chains
$[2*(a_4-1), a_3, a_1, 2*(a_2-1)]$ and
$[2*(a_3-1), a_2, a_4, 2*(a_1-1)]$. In particular, these surfaces
include all the surfaces $S^*(a_1,a_2,a_3,a_4)$ with
$\gcd(w_1,w_3)=\gcd(w_2,w_4)=1$ by \cite{urzua2016characterization-Kollar}.

\begin{theorem}\label{thm:special-surfaces}
  Let
  \begin{eqnarray*}
    A &=& a_1a_2a_3a_4 -a_2a_3a_4 -a_1a_3a_4 -a_1a_2a_4 -a_1a_2a_3 + \\
    && a_1a_2 +a_2a_3 +a_3a_4 +a_1a_4 -a_1 -a_2 -a_3 -a_4 +3 \\
    B_1 &=& a_1a_2a_3a_4 -a_1a_3a_4 -a_1a_2a_3 +a_2a_3 +a_1a_4 -a_1 -a_3 +1\\
    B_2 &=& a_1a_2a_3a_4 -a_2a_3a_4 -a_1a_2a_4 +a_1a_2 +a_3a_4 -a_2 -a_4 +1
  \end{eqnarray*}
Then the following is true:
\begin{enumerate}
\item The surface $T(a_1,a_2,a_3,a_4)$ has ample canonical class $K_T$
  iff $A>0$. 
\item The determinants of the two singularities are $B_1$ and $B_2$. 
\item $K_T^2 = A^2 / B_1B_2.$
\item The minimum $K_T^2 = \frac1{6351}$ is achieved for $(a_i)=(2,2,4,10)$,
  up to a cyclic rotation. 
\end{enumerate}
\end{theorem}


\begin{proof}
  (1) We compute $\pi^*K_X\cdot C$ for a $(-1)$-curve $C$ by
  Lemma~\ref{lem:EK-K2}(1) and find that it is a product of $A$
  and some positive terms.

  (2) is \cite[Lemma 2.4]{hwang2012construction-singular}.

  (3) follows by a direct computation, applying
  Lemma~\ref{lem:EK-K2}(2).
  
  (4) It is somewhat more convenient to use the variables
  $x_i=a_i-1$. Then
  \begin{displaymath}
    \frac{A}{x_1x_2x_3x_4} =
    1- \frac1{x_1}\frac1{x_3}\left( 1+\frac1{x_2}+\frac1{x_4} \right)
    - \frac1{x_2}\frac1{x_4}\left( 1+\frac1{x_1}+\frac1{x_3} \right).
  \end{displaymath}
  One easily checks that for $j=1,2$ the partial derivatives
  $\partial (A/B_j) / \partial x_i\ge0$ when $A>0$ and
  $x_i\ge1$. Thus, it is sufficient to check the minimal collections
  $(a_i)$ for which $A>0$, meaning: for any other
  collection $(a_i')$ with $A>0$ one has $a_i'\ge a_i$ $\forall i$.

  We first find the ``critical'' collections, for which $A=0$. These
  are $(3,3,3,3)$, $(2,8,3,3)$, $(2,3,\frac{11}{2},3)$, $(2,3,3,8)$,
  $(2,2,4,9)$, $(2,2,5,6)$, $(2,2,6,5)$, $(2,2,9,4)$.
  
  Then, modulo rotational symmetry, the smallest collections $(a_i)$
  for which $A>0$ are
 $(4, 3, 3, 3)$,
 $(2, 9, 3, 3)$,
 $(2, 8, 4, 3)$,
 $(2, 8, 3, 4)$,
  $(2, 3, 6, 3)$,
 $(2, 4, 3, 8)$,
 $(2, 3, 4, 8)$,
 $(2, 3, 3, 9)$,
 $(3, 2, 4, 9)$,
 $(2, 3, 4, 9)$,
 $(2, 2, 5, 9)$,
 $(2, 2, 4, 10)$,
 $(3, 2, 5, 6)$,
 $(2, 3, 5, 6)$, \linebreak
 $(2, 2, 6, 6)$, 
 $(2, 2, 5, 7)$,
 $(3, 2, 6, 5)$,
 $(2, 3, 6, 5)$,
 $(2, 2, 7, 5)$,
 $(2, 2, 6, 6)$,
 $(2, 3, 9, 4)$, \linebreak
 $(2, 2, 10, 4)$,
 $(2, 2, 9, 5)$.
Among these, the minimal value $A^2/B_1B_2 = \frac1{6351}$ is
  achieved for $(a_1,a_2,a_3,a_4)=(2,2,4,10)$.
\end{proof}

\bigskip

For the surfaces with boundary, we found a pair with
$(K_X+B_0)^2 = \frac1{78}$. The marks of the corners are 1, 3, $2'$,
$2''$ ($B_0$ goes first, we use the notation $2',2''$ to distinguish
the two vertices with the same marks), and the curves along the edges
have marks 1--3, 1--2--2--1--$2'$, 1--2--1--$2''$, 3--$2'$,
3--1--2--2--2--$2''$, $2'$--$2''$. The weights that work for
Lemma~\ref{lem:near-CY}(3) are $(1,1,2,3)$, $n=7$, $\delta_1=\frac17$,
and $\epsilon_1=1-\frac12-\frac13-\frac1{13}=\frac7{78}$.

\section{Why only four lines?}
\label{sec:four-lines}

It may seem naive and insufficient in search of examples to reduce
oneself only to the simplest of line arrangements: four lines in the
plane. Why not consider some more interesting configurations, e.g. a
Fano or anti-Fano configuration of 7 lines or Segre (resp. dual Segre)
configuration of 12 (resp. 9) lines? And why lines and not conics or curves
of higher degree?  In fact, there are good ad hoc reasons for this:

(1) For all examples of log surfaces arising from 4 lines, one
\emph{apriori} has $K_X^2 \le (K_{\bP^2} + \sum_{i=0}^3L_i)^2 =
1$.
Similarly, for $d$ lines in general position an upper bound is
$(d-3)^2$.  For special line arrangements the upper bound is smaller
but it starts with 2 for a special configuration of 5 lines.  Although
this is an \emph{upper} and not a lower bound, it shows how hard one
has to work to achieve a minimum. Indeed, for $d\ge7$ lines in general
position it is easy to show that
\begin{displaymath}
K_X^2 \ge 9-\frac{d(d-1)}{2} + \frac{d-4}{d-2}(d-4)d = \frac{d^2}{2}
-\frac{11d}{2} +13 +\frac{8}{d-2} \ge \frac35,  
\end{displaymath}
with the minimum achieved by blowing up all of the $\frac{d(d-1)}{2}$
intersection points of the $d$ lines.


(2) The combinatorial game, similar to the one we described in
Section~\ref{sec:game}, becomes very hard to play for more than 4
lines. E.g., for 5 lines the condition on the weights becomes
$w_0+\dotsb +w_4=2n$, and there are very few interesting examples. A
similar thing happens if one works with conics instead of lines. 

(3) We also note that constructions of many interesting log
surfaces with ample $K_X$ can be reduced to
blowups of the same 4-line configuration in $\bP^2$, even when the
initial definition is different, see e.g. \cite{kollar2008is-there,
  hwang2011maximum-number, hwang2012construction-singular,
  urzua2016characterization-Kollar}.
The surfaces of \cite{hwang2012construction-singular}
that use conics and cubics all have bigger volumes.

\section{Lower bound for $K^2(\cS_1) = K^2(\cS_0)$.}
\label{sec:k2-lower-bound}

In this section, we spell out the explicit effective lower bound for
$K^2(\cS_0)$ provided by Theorem 4.8 of
\cite{alexeev2004bounding-singular}. In our present notations, it says
the following:
\begin{displaymath}
  K_X^2 \ge \frac1{\ell \cdot (2\ell)^N}, 
  \quad\text{where } 
  N=128\ell^5+4\ell
  \text{ and }
  \ell = \ulc 1 / {\delta_1(\cS_1)} \urc.
\end{displaymath}
Together with Koll\'ar's bound $\delta_1(\cS_1)\ge \frac1{42}$, this
gives
\begin{displaymath}
  K_X^2 \ge \frac 1 {42 \cdot 84^{128\cdot 42^5 + 4\cdot 42}} \approx
  10^{ - 3.22 \cdot 10^{10}}. 
\end{displaymath}
Certainly this is not a realistic bound. Many improvements can be made
to the estimates in \cite{alexeev2004bounding-singular} but they would
not cardinally change the estimate without introducing some cardinally
new methods. The true lower bound for $K^2(\cS_0)$ may be closer to
Koll\'ar's conjectural bound for
$K^2(\cS_2) = \frac1{(42\cdot 43)^2}$.  Indeed, we dare to think that
it could be close, or equal to
$\frac1{\num{48983}} \approx \frac{67} {(42\cdot 43)^2}$ that we give
here.


\def\cprime{$'$}
\providecommand{\bysame}{\leavevmode\hbox to3em{\hrulefill}\thinspace}
\providecommand{\MR}{\relax\ifhmode\unskip\space\fi MR }
\providecommand{\MRhref}[2]{%
  \href{http://www.ams.org/mathscinet-getitem?mr=#1}{#2}
}
\providecommand{\href}[2]{#2}

\end{document}